\documentclass[portuges,12pt,letter]{article}
\usepackage[centertags]{amsmath}
\usepackage{amsfonts}
\usepackage{newlfont}
\usepackage{amscd}
\usepackage{graphics}
\usepackage{epsfig}
\usepackage{indentfirst}
\usepackage{amsxtra}
\usepackage[latin1]{inputenc}
\usepackage{amssymb, amsmath}
\usepackage{amsthm}
\usepackage[mathscr]{eucal}

\newtheorem{thm}{Theorem}[section]

\author{Fabio Silva Botelho \\ Department of Mathematics \\  Federal University of Santa Catarina, UFSC \\
Florian\'{o}polis, SC - Brazil}

\title{\bf  A  primal dual variational formulation suitable for a large class of non-convex problems in optimization}
\date{}
\begin{document}
\maketitle

\abstract{
In this article we develop a new primal dual variational formulation suitable for a large class of non-convex problems in the calculus of variations.

The results are obtained through basic tools of convex analysis, duality theory, the Legendre transform concept and the respective relations between the primal and dual variables. The novelty here is that the dual formulation is established also for the primal variables, however with a large domain region of concavity about a
critical point.

Finally, we formally prove there is no duality gap between the primal and dual formulations  in a local extremal context.
 }

\section{Introduction} In this work we develop a primal dual variational formulation for non-convex problems in the calculus of
variations  which, in some sense, generalizes, extends and complements the original Telega, Bielski and their co-workers results in the articles \cite{2900,85,11,10}.

The convex analysis results here used may be found in \cite{[6],12,29,120}, for example. Similar results for other problems may be found in \cite{120}.

Finally, details on the function spaces addressed may be found in \cite{1}.

At this point we start to describe the primal formulation.

Let $\Omega \subset \mathbb{R}^3$ be an open, bounded and connected set with a regular (Lipschitzian) boundary denoted by
$\partial \Omega.$

Consider the functional $J:V \rightarrow \mathbb{R}$ where
\begin{eqnarray}
J(u)&=& G_0(\nabla u)+G_1(u)+G_2(u) \nonumber \\ &&- F(u)-\langle u,f \rangle_{L^2}+G_3(u),
\end{eqnarray}
where
$$G_0(\nabla u)=\frac{\gamma}{2}\int_\Omega \nabla u\cdot \nabla u\;dx,$$
$$G_1(u)=\frac{\alpha}{4}\int_\Omega u^4\;dx,$$
$$G_2(u)=\frac{(-\beta+K-\varepsilon)}{2} \int_\Omega u^2\;dx,$$
$$F(u)=\frac{K}{2}\int_\Omega u^2\;dx$$ and
$$G_3(u)=\frac{\varepsilon}{2}\int_\Omega u^2\;dx$$ so that
\begin{eqnarray}
J(u)&=& \frac{\gamma}{2}\int_\Omega \nabla u\cdot \nabla u\;dx+\frac{\alpha}{4}\int_\Omega u^4\;dx \nonumber \\ &&
-\frac{\beta}{2}\int_\Omega u^2\;dx-\langle u,f \rangle_{L^2},\; \forall u \in V.
\end{eqnarray}
Here $dx=dx_1\;dx_2\;dx_3$, $\alpha>0,\beta>0,\gamma>0, \varepsilon>0, \;K>\beta +\varepsilon$, $f \in C(\overline{\Omega})$ and $$V=\{u \in C^2(\overline{\Omega})\;:\; u=0, \text{ on } \partial \Omega\}.$$

Moreover, we recall that $V$ is a Banach space with the norm $\|\cdot \|_V$, where
\begin{eqnarray}\|u\|_V&=&\max_{\mathbf{x} \in \overline{\Omega}}\{|u(\mathbf{x})|+|u_x(\mathbf{x})|+|u_y(\mathbf{x})|+|u_z(\mathbf{x})|
+|u_{xy}(\mathbf{x})|+|u_{xz}(\mathbf{x})|+|u_{yz}(\mathbf{x})| \nonumber \\ &&+|u_{xx}(\mathbf{x})|+|u_{yy}(\mathbf{x})|+|u_{zz}(\mathbf{x})|\}, \forall u \in V,\end{eqnarray}
and generically we denote $$\langle u,v\rangle_{L^2}=\int_\Omega u\;v\;dx, \; \forall u, v \in L^2(\Omega)\equiv L^2,$$
and $$\langle \mathbf{u},\mathbf{v}\rangle_{L^2}=\int_\Omega \mathbf{u} \cdot \mathbf{v}\;dx, \; \forall \mathbf{u}, \mathbf{v} \in L^2(\Omega;\mathbb{R}^3)\equiv L^2.$$

Now observe that
\begin{eqnarray}
\inf_{u \in U}J(u) &\leq&  \inf_{u \in U} \{  G_0(\nabla u)-\langle \nabla u,z_0^* \rangle_{L^2} \nonumber \\ &&
+G_1(u)-\langle u,z_1^*\rangle_{L^2} \nonumber \\ && +G_2(u)-\langle u,z_2^*\rangle_{L^2}\nonumber \\ &&
+G_3(u)-\langle u,f\rangle_{L^2} \nonumber \\ &&+ \sup_{u \in U}\{\langle \nabla u,z_0^* \rangle_{L^2} +
\langle u,z_1^*\rangle_{L^2} +\langle u,z_2^*\rangle_{L^2}-F(u)\}\} \nonumber \\ &=&
 \inf_{u \in U} \{  G_0(\nabla u)-\langle \nabla u,z_0^* \rangle_{L^2} \nonumber \\ &&
+G_1(u)-\langle u,z_1^*\rangle_{L^2} \nonumber \\ && +G_2(u)-\langle u,z_2^*\rangle_{L^2}\nonumber \\ &&
+G_3(u)-\langle u,f\rangle_{L^2}\} \nonumber \\ &&+F^*(-\text{ div } z_0^*+z_1^*+ z_1^*)
\nonumber \\ &=& \sup_{v^* \in Y_1 \times V\times V}\{-G_0^*(v_0^*+z_0^*)-G_1^*(v_1^*+z_1^*) \nonumber \\ &&-G_2^*(v_2^*+z_2^*)-G_3^*(\text{ div }v^*_0-v_1^*-v_2^*+f)\}
\nonumber \\ &&+F^*(-\text{ div } z_0^*+z_1^*+ z_2^*),\;\; \forall z^*=(z_0^*,z_1^*,z_2^*) \in Y_1 \times V \times V,
\end{eqnarray}
where $$F^*(-\text{ div } z_0^*+z_1^*+ z_2^*)=\frac{1}{2K}\int_\Omega (-\text{ div } z_0^*+z_1^*+ z_2^*)^2\;dx,$$ $v^*=(v_0^*,v_1^*,v_2^*) \in Y_1\times V\times V$, $\;Y_1=C^1(\overline{\Omega};\mathbb{R}^3)$ and

\begin{eqnarray}
G_0^*(v_0^*,z_0^*)&=&\sup_{v_0 \in Y_1}\{\langle v_0,v_0^*+z_0^*\rangle_{L^2}-G_0(v_0)\}
\nonumber \\ &=& \sup_{v_0 \in Y_1}\left\{\langle v_0,v_0^*+z_0^*\rangle_{L^2}-\frac{\gamma}{2}\int_\Omega |v_0|^2\;dx\right\}
\nonumber \\ &=& \frac{1}{2\gamma}\int_\Omega |v_0^*+z_0^*|^2\;dx.
\end{eqnarray}

 Also,
\begin{eqnarray}
G_1^*(v_1^*,z_1^*)&=&\sup_{u \in V}\{\langle u ,v_1^*+z_1^*\rangle_{L^2}-G_1(u)\}
\nonumber \\ &=& \sup_{u \in V}\left\{\langle u,v_1^*+z_1^*\rangle_{L^2}-\frac{\alpha}{4}\int_\Omega u^4\;dx\right\}
\nonumber \\ &=& \frac{3}{4\alpha^{1/3}}\int_\Omega |v_1^*+z_1^*|^{4/3}\;dx,
\end{eqnarray}
\begin{eqnarray}
G_2^*(v_2^*,z_2^*)&=&\sup_{u \in V}\{\langle u ,v_2^*+z_2^*\rangle_{L^2}-G_2(u)\}
\nonumber \\ &=& \sup_{u \in V}\left\{\langle u,v_2^*+z_2^*\rangle_{L^2}-\frac{(-\beta+K-\varepsilon)}{2}\int_\Omega u^2\;dx\right\}
\nonumber \\ &=& \frac{1}{2(K-\beta-\varepsilon)}\int_\Omega (v_2^*+z_2^*)^2\;dx
\end{eqnarray}
and
\begin{eqnarray}
G_3^*(\text{ div }v^*_0-v_1^*-v_2^*+f)&=&\sup_{u \in V}\{-\langle \nabla u ,v_0^*\rangle_{L^2}-\langle u,v_1^*\rangle_{L^2}
\nonumber \\ &&-\langle u,v_2^* \rangle_{L^2}+\langle u,f \rangle_{L^2}-G_3(u)\}
\nonumber \\ &=& \sup_{u \in V}\{-\langle \nabla u ,v_0^*\rangle_{L^2}-\langle u,v_1^*\rangle_{L^2}
\nonumber \\ &&-\langle u,v_2^* \rangle_{L^2}+\langle u,f \rangle_{L^2}-\frac{\varepsilon}{2}\int_\Omega u^2\;dx\}
\nonumber \\ &=& \frac{1}{2\varepsilon} \int_\Omega(\text{ div }v^*_0-v_1^*-v_2^*+f)^2\;dx.
\end{eqnarray}
At this point, we denote,
\begin{eqnarray}\label{p0a121}
J^*(v^*,z^*)&=& -G_0^*(v_0^*+z_0^*)-G_1^*(v_1^*+z_1^*) \nonumber \\ &&-G_2^*(v_2^*+z_2^*)-G_3^*(\text{ div }v^*_0-v_1^*-v_2^*+f)
\nonumber \\ && +F^*(-\text{ div } z_0^*+z_1^*+ z_2^*).
\end{eqnarray}

The extremal equation,
$$\frac{\partial J^*(v^*,z^*)}{\partial z^*}=\mathbf{0}$$ gives the following system.

Specifically from
$$\frac{\partial J^*(v^*,z^*)}{\partial z^*_0}=\mathbf{0},$$ we get
$$-\frac{\partial G_0^*(v_0^*+z_0^*)}{\partial z_0^*}+\nabla \left(\frac{\partial F^*(-\text{ div } z_0^*+z_1^*+ z_2^*)}{\partial w_0^*}\right)=\mathbf{0},$$
where $$w_0^*=-\text{ div }z_0^*$$ so that

$$\frac{v_0^*+z_0^*}{\gamma}-\nabla \left(\frac{-div z_0^*+z_1^*+z_2^*}{K}\right)=\mathbf{0}.$$

Hence, defining $$\hat{u}=\frac{-div z_0^*+z_1^*+z_2^*}{K},$$ we have
\begin{equation}\label{p0a122}v_0^*=-z_0^*+\gamma \nabla \hat{u}.\end{equation}

From
$$\frac{\partial J^*(v^*,z^*)}{\partial z^*_1}=\mathbf{0},$$ we get
$$-\frac{\partial G_1^*(v_1^*+z_1^*)}{\partial z_1^*}+\frac{\partial F^*(-\text{ div } z_0^*+z_1^*+ z_2^*)}{\partial z^*_1}=\mathbf{0},$$
so that
$$\frac{(v_1^*+z_1^*)^{1/3}}{\alpha^{1/3}}=\hat{u},$$
that is,
\begin{equation}\label{p0a123}v_1^*=-z_1^*+\alpha \hat{u}^3.\end{equation}

Finally, from
$$\frac{\partial J^*(v^*,z^*)}{\partial z^*_2}=\mathbf{0},$$ we get
$$-\frac{\partial G_2^*(v_2^*+z_2^*)}{\partial z_2^*}+\frac{\partial F^*(-\text{ div } z_0^*+z_1^*+ z_2^*)}{\partial z^*_2}=\mathbf{0},$$
so that
$$\frac{(v_2^*+z_2^*)}{K-\beta-\varepsilon}=\hat{u},$$
that is,
\begin{equation}\label{p0a124}v_2^*=-z_2^*+(K-\beta-\varepsilon) \hat{u}.\end{equation}

Replacing (\ref{p0a122}), (\ref{p0a123}) and (\ref{p0a124}) into (\ref{p0a121}), we obtain
\begin{eqnarray}
J^*(v^*,z^*) &=& -\frac{\gamma}{2} \int_\Omega |\nabla \hat{u}|^2\;dx-\frac{3}{4\alpha^{1/3}}\int_\Omega \alpha^{4/3} \hat{u}^4\;dx
\nonumber \\ &&-\frac{(K-\beta-\varepsilon)}{2}\int_\Omega \hat{u}^2\;dx+\frac{K}{2}\int_\Omega \hat{u}^2\;dx
\nonumber \\ &&-\frac{1}{2\varepsilon} \int_\Omega (\text{ div }z_0^*-z_1^*-z_2^*-\gamma \nabla^2\hat{u}+\alpha \hat{u}^3+(K-\beta-
\varepsilon)\hat{u}-f)^2
\;dx \nonumber \\ &=& -\frac{\gamma}{2} \int_\Omega |\nabla \hat{u}|^2\;dx-\frac{3\;\alpha}{4}\int_\Omega  \hat{u}^4\;dx
+\frac{\beta+\varepsilon}{2}\int_\Omega \hat{u}^2\;dx
\nonumber \\ &&-\frac{1}{2\varepsilon} \int_\Omega (-\gamma \nabla^2\hat{u}+\alpha \hat{u}^3-(\beta+
\varepsilon)\hat{u}-f)^2
\;dx  \nonumber \\ &\equiv& \hat{J}_\varepsilon(\hat{u}).
\end{eqnarray}
\section{ The main duality principle}
With such statements, definitions  and results in mind we prove the following theorem.
\begin{thm} Considering the context of the last section statements, definitions and results, let $J:V \rightarrow \mathbb{R}$ be defined by
\begin{eqnarray}
J(u)&=& \frac{\gamma}{2}\int_\Omega \nabla u\cdot \nabla u\;dx+  \frac{\alpha}{4}\int_\Omega u^4\;dx \nonumber \\ &&
-\frac{\beta}{2}\int_\Omega u^2\;dx-\langle u,f \rangle_{L^2},\; \forall u \in V.
\end{eqnarray}
Here $f \in C(\overline{\Omega}).$
Let $\hat{J}_\varepsilon:V \rightarrow \mathbb{R}$ be defined by
\begin{eqnarray} \hat{J}_\varepsilon(\hat{u})&=&  -\frac{\gamma}{2} \int_\Omega |\nabla \hat{u}|^2\;dx-\frac{3\;\alpha}{4}\int_\Omega  \hat{u}^4\;dx
+\frac{\beta+\varepsilon}{2}\int_\Omega \hat{u}^2\;dx
\nonumber \\ &&-\frac{1}{2\varepsilon} \int_\Omega (-\gamma \nabla^2\hat{u}+\alpha \hat{u}^3-(\beta+
\varepsilon)\hat{u}-f)^2
\;dx.
\end{eqnarray}

Assume $u_0 \in V$ is such that $$\delta J(u_0)=\mathbf{0},$$ and $$\delta^2J(u_0) > \mathbf{0}.$$

Under such hypotheses, $$\delta \hat{J}_\varepsilon(u_0)=\mathbf{0},$$ $$J(u_0)=\hat{J}_\varepsilon(u_0)$$ and for $\varepsilon>0$ sufficiently small,
\begin{eqnarray}\delta^2 \hat{J}_\varepsilon(u_0)&=& -(\delta^2 J(u_0)-\varepsilon)-\frac{(\delta^2J(u_0)-\varepsilon)^2}{\varepsilon}
\nonumber \\ &\approx& -\mathcal{O}\left(\frac{1}{\varepsilon}\right) \nonumber \\ &<& \mathbf{0},\end{eqnarray} so that there exist $r,r_1>0$ such that
\begin{eqnarray}
J(u_0)&=&\inf_{u \in B_r(u_0)}J(u) \nonumber \\ &=&
\sup_{\hat{u} \in B_{r_1}(u_0)}\hat{J}_\varepsilon(\hat{u})
\nonumber \\ &=&\hat{J}_\varepsilon(u_0).
\end{eqnarray}
\end{thm}
\begin{proof}
From $\delta J(u_0)=\mathbf{0}$ we have
$$-\gamma \nabla^2u_0+\alpha u_0^3-\beta u_0-f=0, \text{ in } \Omega.$$

Thus defining $\hat{u}=u_0$ we have
\begin{eqnarray}&&-\gamma \nabla^2u_0+\alpha u_0^3-(\beta +\varepsilon)u_0-f
\nonumber \\ &=& -\varepsilon u_0 \nonumber \\ &=&-\varepsilon \hat{u}, \text{ in } \Omega.\end{eqnarray}

Hence
$$\hat{u}=-\frac{-\gamma \nabla^2u_0+\alpha u_0^3-(\beta +\varepsilon)u_0-f}{\varepsilon}.$$

From such an expression for $\hat{u}$ we obtain
\begin{eqnarray}
\delta\hat{J}_\varepsilon(u_0)=\gamma \nabla^2u_0-3\alpha u_0^3+(\beta+\varepsilon)u_0-
\gamma \nabla^2 \hat{u}+3 \alpha \hat{u}u_0^2-(\beta+\varepsilon)\hat{u}.
\end{eqnarray}

From this, since $\hat{u}=u_0$, we get
$$\delta \hat{J}_\varepsilon(u_0)=\mathbf{0}.$$

Also, we may define $v^*,z^*$  such that
\begin{equation}\label{p0a122a}v_0^*=-z_0^*+\gamma \nabla \hat{u},\end{equation}
\begin{equation}\label{p0a123b}v_1^*=-z_1^*+\alpha \hat{u}^3,\end{equation}
\begin{equation}\label{p0a124c}v_2^*=-z_2^*+(K-\beta-\varepsilon) \hat{u},\end{equation}
$$-\text{ div } z_0^*+z_1^*+z_2^*=K\hat{u},$$
so that
\begin{eqnarray}\text{ div }v_0^*-v_1^*-v_2^*&=& -\text{ div }z_0^*+z_1^*+z_2^* \nonumber \\ &&
+\gamma \nabla^2 \hat{u}-\alpha \hat{u}^3-(K-\beta-\varepsilon) \hat{u} \nonumber \\ &=& \gamma \nabla^2 \hat{u}-\alpha \hat{u}^3+(\beta+\varepsilon) \hat{u}
\nonumber \\ &=& \varepsilon \hat{u}-f.
\end{eqnarray}
From such relations we obtain,
$$G_0^*(v_0^*,z_0^*)=\langle \nabla \hat{u},v_0^*+z_0^* \rangle_{L^2}-G_0(\nabla \hat{u}),$$
$$G_1^*(v_1^*,z_1^*)=\langle  \hat{u},v_1^*+z_1^* \rangle_{L^2}-G_1(\hat{u}),$$
$$G_2^*(v_2^*,z_2^*)=\langle  \hat{u},v_2^*+z_2^* \rangle_{L^2}-G_3(\hat{u}),$$
$$G_3^*(\text{ div }v^*_0-v_1^*-v_2^*+f)=\langle \hat{u}, \text{ div }v^*_0-v_1^*-v_2^* \rangle_{L^2}+\langle \hat{u},f\rangle_{L^2}-G_3(\hat{u})$$
and
$$F^*(\text{ div }z^*_0-z_1^*-z_2^*)=\langle \hat{u},- \text{ div }z^*_0+z_1^*+z_2^* \rangle_{L^2}-F(\hat{u}).$$

From such results, we obtain
\begin{eqnarray}
\hat{J}_\varepsilon(\hat{u})&=& J^*(v^*,z^*)\nonumber \\ &=& -G_0^*(v_0^*+z_0^*)-G_1^*(v_1^*+z_1^*) \nonumber \\ &&-G_2^*(v_2^*+z_2^*)-G_3^*(\text{ div }v^*_0-v_1^*-v_2^*+f)
+F^*(-\text{ div } z_0^*+z_1^*+ z_2^*) \nonumber \\ &=& G_0(\nabla \hat{u})+G_1(\hat{u})+G_2(\hat{u})+G_3(\hat{u})-F(\hat{u}) -\langle \hat{u},f \rangle_{L^2}
\nonumber \\ &=&
J(\hat{u})\nonumber \\ &=& J(u_0). \end{eqnarray}

Finally, for $\varepsilon>0$ sufficiently small,\begin{eqnarray}
\delta^2\hat{J}_\varepsilon(u_0)&=& \gamma \nabla^2-9\alpha u_0^2+\beta+\varepsilon \nonumber \\ &=&
-\frac{(-\gamma \nabla^2+3\alpha u_0^2-(\beta+\varepsilon))^2}{\varepsilon}+6\alpha u_0^2 \nonumber \\ &=&
\gamma \nabla^2-3\alpha u_0^2+\beta+\varepsilon-\frac{(-\gamma \nabla^2+3\alpha u_0^2-(\beta+\varepsilon))^2}{\varepsilon} \nonumber \\
 &=&-(\delta^2 J(u_0)-\varepsilon)-\frac{(\delta^2J(u_0)-\varepsilon)^2}{\varepsilon} \nonumber \\
 &\approx&
-\mathcal{O}\left(\frac{1}{\varepsilon}\right) \nonumber \\ &<& \mathbf{0}.
\end{eqnarray}
From these last results, there exist $r,r_1>0$ such that
\begin{eqnarray}
J(u_0)&=&\inf_{u \in B_r(u_0)}J(u) \nonumber \\ &=&
\sup_{\hat{u} \in B_{r_1}(u_0)}\hat{J}_\varepsilon(\hat{u})
\nonumber \\ &=&\hat{J}_\varepsilon(u_0).
\end{eqnarray}

The proof is complete.
\end{proof}

\section{Conclusion}
In this article we develop a primal dual variational formulation for a large class of problems in the calculus of variations.

We emphasize the dual functional obtained has a large domain region of concavity about a critical point, which makes such a formulation
 very interesting from a numerical analysis point of view.

Finally, it has been formally proven  there is no duality gap between the primal and dual formulations in a local extremal context.

\end{document}